\documentclass[12pt]{amsart}
\usepackage{amssymb,url,color,enumerate}
\usepackage[left=2cm,right=2cm,top=2cm,bottom=2cm]{geometry}
\usepackage{hyperref}
\newtheorem{theorem}{Theorem}

\newtheorem{corollary}[theorem]{Corollary}

\newtheorem{lemma}[theorem]{Lemma}

\newtheorem{remark}[theorem]{Remark}

\theoremstyle{definition}
\newtheorem{example}[theorem]{Example}

\let\Re\relax
\DeclareMathOperator{\Re}{Re}

\newcommand{\C}{\mathbb{C}}

\newcommand{\N}{\mathbb{N}}

\newcommand{\Q}{\mathbb{Q}}
\newcommand{\R}{\mathbb{R}}
\newcommand{\Z}{\mathbb{Z}}

\newcommand{\re}{\mathrm{Re}}
\newcommand{\im}{\mathrm{Im}}

\newcommand{\s}{\mathcal{S}}
\newcommand{\sh}{\mathcal{S}^{\sharp}}
\newcommand{\shf}{\mathcal{S}^{\sharp\flat}}

\newcommand{\shfs}{\mathcal{S}^{\sharp \flat }(\sigma_0, \sigma_1)}
\newcommand{\abs}[1]{\left\vert#1\right\vert}
\newcommand{\lj}{\lambda_{j}}
\newcommand{\mj}{\mu_{j}}

\usepackage{amssymb}
\usepackage{graphicx}
\catcode`\é=\active
\defé{\'e}

\catcode`\è=\active
\defè{\`e}

\catcode`\ê=\active
\defê{\^e}

\catcode`\û=\active
\defû{\^u}

\catcode`\â=\active
\defâ{\^a}

\catcode`\à=\active
\defà{\`a}

\catcode`\ù=\active
\defù{\`u}

\catcode`\ç=\active
\defç{\c c}

\catcode`\î=\active
\defî{\^{\i}}

\catcode`\ï=\active
\defï{\"{\i}}

\catcode`\ô=\active
\defô{\^o}

\catcode`\é=\active
\defé{\'e}

\title[On relations equivalent to the GRH for the Selberg class]{On relations equivalent to the generalized Riemann hypothesis for the Selberg class}

\author{Kamel mazhouda and Lejla Smajlovi\'c}
\address{Kamel Mazhouda\\
Faculty of Science of Monastir\\Department of Mathematics\\5000
Monastir\\Tunisia} \email{kamel.mazhouda@fsm.rnu.tn}
\address{Lejla Smajlovi\'c\\ Department of Mathematics\\ University of Sarajevo\\ Zmaja od Bosne 33-35\\ Sarajevo, Bosnia and Herzegovina}
\email{lejla.smajlovic@efsa.unsa.ba}
\date{\today}

\begin{document}

\thanks{The first author is supported by the Tunisian-French Grant DGRST-CNRS 14/R 1501.}



\begin{abstract}
In this paper we prove that the Generalized Riemann Hypothesis (GRH) for functions in the class $\shf$ containing the Selberg class is equivalent to a certain integral expression
of the real part of the generalized Li coefficient $\lambda_F(n)$ associated to $F\in\shf$, for positive integers $n$. Moreover, we deduce that the GRH is equivalent to a certain expression of $\re(\lambda_F(n))$ in terms of the sum of the Chebyshev polynomials of the first kind. Then, we partially evaluate the integral expression and deduce further relations equivalent to the GRH involving the generalized Euler-Stieltjes constants of the second kind associated to $F$. The class $\shf$ unconditionally contains all automorphic $L$-functions attached to irreducible cuspidal unitary representations of $\mathrm{GL}_N(\Q)$, hence, as a corollary we also derive relations equivalent to the GRH for automorphic $L$-functions.

\end{abstract}

\maketitle
\tableofcontents

\section{Introduction}\label{se.1}

The Selberg class of functions $\s$, introduced by A. Selberg in
\cite{Selberg}, is a general class of Dirichlet series $F$
satisfying the following properties:

\begin{itemize}
\item[(i)]  (Dirichlet series) $F$ possesses a Dirichlet series representation
\begin{equation} \label{ds rep}
F(s)=\sum_{n=1}^{\infty}\frac{a_{F}(n)}{n^{s}},
\end{equation}
that converges absolutely for $\re (s)>1.$

\item[(ii)]  (Analytic continuation) There exists an integer $m \geq 0$ such
that $(s-1)^{m}F(s)$ is an entire function of finite order. The smallest
such number is denoted by $m_{F}$ and called the polar order of $F$.

\item[(iii)]  (Functional equation) The function $F$ satisfies the
functional equation
\begin{equation*}
\xi_{F}(s)=w\overline{\xi_{F}(1-\bar{s})},
\end{equation*}
where
\begin{equation*}
\xi_{F}(s)=s^{m_F}(1-s)^{m_F}F(s)Q^{s}_{F}\prod_{j=1}^{r}\Gamma(\lj s+\mj),
\end{equation*}
with $Q_{F}>0$, $r\geq 0$, $\lambda_{j}>0$, $\abs{w}=1$,
$\re\mj\geq 0$, $j=1,\ldots ,r.$

\item[(iv)]  (Ramanujan conjecture) For every $\epsilon >0,\ a_{F}(n)\ll
n^{\epsilon }$.

\item[(v)]  (Euler product)
\begin{equation*}
\log F(s)=\sum_{n=1}^{\infty }\frac{b_{F}(n)}{n^{s}},
\end{equation*}
where $b_{F}(n)=0$, for all $n\neq p^{m}$ with $m\geq 1$ and $p$
prime, and $b_{F}(n)\ll n^{\theta}$, for some
$\theta<\frac{1}{2}$.
\end{itemize}

The extended Selberg class $\sh$, introduced in \cite{Kacz-PerelliActa} is a class of functions
satisfying conditions (i), (ii) and (iii).\\

As usual, the non-trivial zeros of $F\in\s$ are zeros of the completed function $\xi_F$; we denote the set of non-trivial zeros of $F$ by $Z(F)$. One of the most important conjectures about the Selberg class is the Generalized Riemann Hypothesis (GRH), i.e. the conjecture that for all $F\in\s$, the non trivial zeros of $F$ are located on the critical line $\re(s)=\frac{1}{2}$. For more details concerning the Selberg class we refer to the surveys of Kaczorowski \cite{KaczSurvay} and Perelli \cite{Perelli}.\\

The main class of this paper is the class $\shf \supseteq \s$, defined in
Section \ref{sec:preliminaries} below, consisting of functions $F\in\sh$
that possess an Euler
sum. The class $\shf$ is similar to the class
$\widetilde{\mathcal{S}}$ considered in \cite{AkMur2008}, whose
elements under average density hypothesis have uniformly
distributed (modulo 1) imaginary parts of zeros. The main reason why we consider the class $\shf$ is that, as proved in \cite[Prop.
2.1. and Sec. 4]{OdSm11}, it contains both the Selberg class $\s$ and (unconditionally) the
class of all automorphic $L$ functions attached to automorphic
irreducible unitary cuspidal representations of
$GL_{N}(\mathbb{Q})$.\\

For an integer $n\neq0$, the generalized $n$th Li coefficient attached to $F\in\shf$ non-vanishing at zero is defined as the sum
\begin{equation} \label{lambda F def}
\lambda_{F}(n)={\sum_{\rho\in Z(F)}}^{*}\left(1-\left(1-\frac{1}{\rho}%
\right)^{n}\right),
\end{equation}
where the $\ast$ means that the sum is taken in the sense of the limit $\lim\limits_{T\to\infty} \sum_{|\im(\rho)|\leq T}$. In the sequel, we denote by $\shf_0$ the set of all $F\in\shf$, non-vanishing at zero. For $t>0$, by $\shf_{t}$ we denote the set of all $F\in\shf$ such that all eventual non-trivial zeros of $F$ which lie off the critical line have the absolute value of the imaginary part bigger than $t$.

\vskip .06in
The existence of coefficients $\lambda _{F}(n)$, for $n\in\Z\setminus\{0\}$ and $F\in\shf_0$ is proved in \cite[Theorem 4.1.]{Sm10}, together with the generalized Li criterion, stating that the GRH for $F\in\shf_0$ is equivalent to the
non-negativity of the set of numbers $\re(\lambda_{F}(n))$ for $n\in
\mathbb{N}$ (see \cite[Theorem 4.3.]{Sm10}).
\vskip .06in
Moreover, in \cite{OdSm11} and \cite{OM} it is proved that for $F \in \shf_0$
\begin{equation} \label{GRH equiv}
GRH\ \Leftrightarrow\ \lambda_{F}(n)=\frac{d_{F}}{2}n \log n + c_{F}n+ O(\sqrt{n}\log n), \text{  as   } n\to\infty,
\end{equation}
where
$$c_{F}=\frac{d_{F}}{2}(\gamma-1)+\frac{1}{2}\log(\lambda Q_{F}^{2}), \ \ \ \lambda=\prod_{j=1}^{r}\lambda_{j}^{2\lambda_{j}},$$
$d_{F}=2\sum\limits_{j=1}^{r}\lambda_{j}$ is the degree of $F$ and $\gamma$ is the classical Euler constant.

The above result was proved also in \cite{MAZ} using the saddle-point method in conjunction
with the theory of the N\"orlund-Rice integrals.\\

There exist many relations equivalent to the Riemann Hypothesis. An interesting integral criteria, proved by M. Balazard, E. Saias, and M. Yor \cite{BSY99} states that the Riemann Hypothesis is equivalent to the equation
\begin{equation*} \label{BSY crit}
\int\limits_{-\infty}^{\infty} \frac{\log|\zeta(1/2+it)|}{1+4t^2}=0,
\end{equation*}
where $\zeta(s)$ denotes the Riemann zeta function.

Another integral criteria is obtained by V. Volchkov in \cite{Volchkov} and states that the Riemann Hypothesis is equivalent to the equation
\begin{equation}\label{Volch crit}
\int\limits_{0}^{\infty}\left(\frac{1-12t^2}{(1+4t^2)^3}\int\limits_{\frac{1}{2}}^{\infty}\log |\zeta(\sigma + it)|d\sigma \right) dt = \pi \frac{3-\gamma}{32}.
\end{equation}
Both results are generalized by S. K. Sekatskii, S. Beltraminelli and D. Merlini in \cite{SBM12} and \cite{SBM09}.

Actually, an elementary computation, based on application of Littlewood's theorem to the function $\log\zeta(s)$ and the appropriate rectangular contour shows that equation \eqref{Volch crit} is equivalent to the equation
\begin{equation}\label{sbm crit}
\int\limits_{0}^{\infty} \frac{t\arg{\zeta(1/2 +it)}}{(1/4 + t^2)^2}dt=\pi \frac{\gamma - 3}{2},
\end{equation}
hence we call equation \eqref{sbm crit} the Volchkov criterion for the Riemann Hypothesis as well.

The Volchkov criterion and the  results of \cite{SBM09} are further generalized in \cite{HJM15}, where equation \eqref{sbm crit} was interpreted in terms of the argument of the  Veneziano amplitude.
\vskip .06in
In this paper, we prove that the GRH for function $F\in\shf_0$ is equivalent to the integral representation
\begin{equation}\label{int rep re lambda}
\re(\lambda_F(n))= 16n\int\limits_{0}^{\infty} N_F(x) \frac{x}{(4x^2+1)^2} U_{n-1}\left(\frac{4x^2 -1}{4x^2 +1} \right)dx - (1-(-1)^n)N_F(0)
\end{equation}
of the real part of the $n$th generalized Li coefficient $\lambda_F(n)$, for all positive integers $n$. Here, $N_F(x)$ denotes the counting function of the number of non-trivial zeros $\rho\in Z(F)$ such that $|\im(\rho)|\leq x$ and $U_{n-1}(x)$ is the Chebyshev polynomial of the second kind, see sections 2.2. and 2.5. Obviously, $N_F(0)$ is the number of Siegel zeros of $F$, i.e. eventual non-trivial real zeros of $F$.
\vskip .06in
We also prove that the GRH for $F\in\shf_0$ is equivalent to the representation
\begin{equation}\label{Re(lambda) under GRH}
\re(\lambda_F(n))= \sum_{\rho=\sigma + i\gamma\in Z(F)} \left( 1-T_n \left(\frac{4\gamma^2 -1}{4\gamma^2 +1} \right)\right)
\end{equation}
of $\re(\lambda_F(n))$, for all positive integers $n$, where $T_n(x)$ is the Chebyshev polynomial of the first kind and the zeros on the right-hand side of \eqref{Re(lambda) under GRH} are taken according to their multiplicities.

Then, we partially evaluate the integral in \eqref{int rep re lambda} and write $\re(\lambda_F(n))$ as a sum of integral and oscillatory part, which we relate to the generalized Euler-Stieltjes constants of the second kind associated to $F\in\shf_0$. Then, we show that the GRH for $F\in\shf_0$ is equivalent to a certain asymptotic integral formula (see formula \eqref{asymptotic eq to GRH} below).\\

In the case when $n=1$ or $n=2$, under additional assumptions on the location of the imaginary part of the first non-trivial zero of $F\in\shf$, which lies  off the critical line, we prove that equations \eqref{int rep re lambda} and \eqref{Re(lambda) under GRH} with $n=1$ or $n=2$ are equivalent to the GRH.  As a special case of this result and the evaluation of the integral in \eqref{int rep re lambda} when $F=\zeta$ and $n=1$ we deduce equation \eqref{sbm crit}, meaning that the Volchkov criterion for the Riemann Hypothesis is a special case of our results.

Moreover, we prove that the GRH for the function $F\in\shf_{1/\sqrt{3}}$ is equivalent to a certain integral expression of the real part of the constant term in the Laurent (Taylor) series expression of the logarithmic derivative $F'(s)/F(s)$ at $s=1$. \\

The class $\shf$ unconditionally contains the class of $L-$functions attached to irreducible, cuspidal unitary representations of $\mathrm{GL}_N(\Q)$, hence, as a corollary, we also derive relations equivalent to the GRH for automorphic $L-$functions.\\

The paper is organized as follows: in section 2 we recall necessary background results; in section 3 we derive representations of the real part of the generalized Li coefficient equivalent to the GRH. Section 4 is devoted to partial evaluation of the integral on the right hand side of \eqref{int rep re lambda}, while in section 5 we discuss the relation with the generalized Euler-Stieltjes constants of the second kind. In section 6 we derive results for the automorphic $L-$functions. Concluding remarks are presented in the last section.

\section{Preliminaries}\label{sec:preliminaries}

\subsection{The class $\shf$}
Throughout this paper, we focus on the class $\shf$ of
functions satisfying axioms (i), (ii) and the following two
axioms:
\begin{enumerate}
\item[(iii')]  (Functional equation) The function $F$ satisfies the
functional equation
\begin{equation*}
\xi _{F}(s)=w\overline{\xi _{F}(1-\bar{s})},
\end{equation*}
where
\begin{equation*}
\xi _{F}(s)=s^{m_F}(1-s)^{m_F}F(s)Q_{F}^{s}\prod_{j=1}^{r}\Gamma (\lj s+\mj)= \gamma_F(s)F(s)
\end{equation*}
with $Q_{F}>0$, $r\geq 0$, $\lambda _{j}>0$, $\abs{w}=1$,
$\re\mj>-\frac{1}{4}$, $\re(\lj+2\mj)>0$, $j=1,\ldots,r.$

\item[(v')]  (Euler sum) The logarithmic derivative of the function $F$ possesses a
Dirichlet series representation
\begin{equation*}
\frac{F'}{F}(s)=-\sum_{n=2}^{\infty}\frac{c_{F}(n)}{n^{s}}
\end{equation*}
converging absolutely for $\re (s)>1$.
\end{enumerate}

Let us note that (iii') implies that $\re(\lj+\mj)>0$.\\

For $F\in\shf$ it is easy to deduce, using axioms (i), (iii') and the Phragm\' en Lindel\"of principle that the completed function $\xi_F$ is entire function of order one. Moreover, using an explicit formula for the class $\shf$ applied to a suitably chosen test function, it is proved in \cite{OdSm11} that for all $F\in\shf_0$ and all $z\notin Z(F)$ one has
\begin{equation}\label{log der xi F}
\frac{\xi_F'}{\xi_F}(z) = \lim_{T \to\infty} \sum_{\rho \in Z(F), |\im(\rho)|\leq T} \frac{1}{z-\rho} =\left. \sum_{\rho \in Z(F)}\right.^{\ast} \left(\frac{1}{z-\rho}\right) ,
\end{equation}
where each zero is counted according to its multiplicity.

\subsection{Distribution of non-trivial zeros of $F\in\shf$}

For $T>0$, such $T$ and $-T$ are not the ordinates of a non-trivial zero, we denote by $N_F(T)$ the number of non-trivial zeros $\rho=\sigma+it\in Z(F)$ of $F\in\shf$ such that $|t|\leq T$. By $N_F^+(T)$ and $N_F(T)^-$ respectively we denote the number of non-trivial zeros of $F$ with $0\leq\im\rho\leq T$, respectively $-T\leq\im\rho\leq 0$.

In the case when $T$ or $-T$ is the ordinate of a non-trivial zero, we define $N_F(T)=N_F(T+0)$ and $N_F^{\pm}(T)=N_F^{\pm}(T+0)$.\\

Vertical distribution of zeros of a function $F\in\sh$
satisfying axiom (v') was discussed in \cite[Sec.
5.1.]{Sm10} where it was proved that
\begin{equation} \label{N F}
N_{F}^{+}(T)=\frac{d_{F}}{2\pi}T\log T+C_{F}T+ a_F \log T + m_F+O(1/T)+S_{F}^{+}(T), \text{  as  } T\to\infty,
\end{equation}
where
$$C_{F}=\frac{1}{2\pi}(\log q_{F}-d_{F}(\log(2\pi)+1)),$$
$$ q_{F}=(2\pi)^{d_{F}}Q_{F}^{2}\prod_{j=1}^{r}\lambda_{j}^{2\lambda_{j}}$$
is the conductor of $F$,
$$
a_F=\frac{1}{\pi}\sum_{j=1}^{r}\im(\mu_j)
$$
and, in the case when $T$ is not the ordinate of the non-trivial zero, $S_F^{+}(T)$ is the value of the function $\frac{1}{\pi}\arg F(s)$ obtained by continuous variation along the straight lines joining points $2$, $2+iT$ and $1/2 + iT$. When $T$ is the ordinate of a non-trivial zero of $F$, $S_F^+(T)$ is by definition equal to $S_F^{+}(T+0)$. Moreover, it is proved that $S_{F}^{+}(T)= O_{F}(\log T)$, as $T\to\infty$.\\

Proceeding analogously as in \cite{Sm10}, applying the argument principle to the completed function $\xi_F(s)$ and the rectangle with vertices $-1-iT$, $2-iT$, $2+iT$ and $-1+iT$ we get
\begin{multline*}
N_{F}(T)=2m_F + \frac{2T\log Q_F}{\pi} + \\+ \frac{1}{\pi} \sum_{j=1}^{r} \left( \arg \Gamma\left(\lambda_j\left(\frac{1}2 + iT\right) +\mu_j \right) -\arg \Gamma\left(\lambda_j\left(\frac{1}2 - iT\right) +\mu_j \right) \right) +  S_F(T),
\end{multline*}
or, equivalently,
\begin{multline} \label{N F in terms of argument}
N_{F}(T)=2m_F + \frac{2T\log Q_F}{\pi} + \\+ \frac{1}{\pi}\im\left( \sum_{j=1}^{r} \left( \log \Gamma\left(\lambda_j\left(\frac{1}2 + iT\right) +\mu_j \right) +\log \Gamma\left(\lambda_j\left(\frac{1}2 + iT\right) +\overline{\mu_j} \right) \right) \right) +  S_F(T),
\end{multline}
where, in the case when $\pm T$ is not the ordinate of the non-trivial zero, $S_F(T)$ is the value of the function $\frac{1}{\pi}\arg F(s)$ obtained by continuous variation along the straight lines joining points $1/2 -iT$, $2-iT$, $2+iT$ and $1/2 + iT$ and, in the case when $\pm T$ is the ordinate of a non-trivial zero of $F$, we define $S_F(T)$ to be equal to $S_F(T+0)$.

Moreover, it is easy to see that the bound $S_F(x)=O_{F}(\log T)$, as $T\to\infty$ holds true in the case when the axiom (iii) is replaced by (iii').\\

In the case when the coefficients $a_F(n)$ in the Dirichlet series representation \eqref{ds rep} of $F\in\shf$ are real and the function has no Siegel zeros, application of the reflection principle and the functional equation axiom (iii') yields that non-trivial zeros come in conjugate pairs, hence $N_{F}^{+}(T)=N_{F}^{-}(T)$. Moreover, in this case the function $S_F(T)$ is equal to the value of the function $\frac{2}{\pi} \arg F(s)$ obtained by continuous variation along the straight lines joining points $2$, $2+iT$ and $1/2 + iT$, which is actually equal to $\frac{2}{\pi} \arg F(1/2+iT)$.

\subsection{The generalized Li coefficients}

The convergence of the sum \eqref{lambda F def} defining the generalized Li coefficients  associated to the function $F\in\shf_0$ is proved in \cite{Sm10}, using the representation \eqref{log der xi F} of the logarithmic derivative of the completed function $\xi_F$ and the fact that $\xi_F$ is entire function of order one.\\

The symmetry $\rho \leftrightarrow (1-\overline{\rho})$
in the set $Z(F)$ of non-trivial zeros of $F
\in\mathcal{S}^{\sharp\flat}_0$ implies that
$\lambda_{F}(-n)=\overline{\lambda_{F}(n)}$, for all
$n\in\mathbb{N}$, hence, $\re(\lambda_{F}(n))=\re(\lambda_{F}(-n))$,
for all $n\in\mathbb{N}$.\\

The importance of the generalized Li coefficients lies in the fact that the GRH  for  $F\in\shf_0$ is equivalent to positivity of the sequence of real numbers $\{\re(\lambda_{F}(n))\}_{n\geq 1}$. Actually, the GRH for  $F\in\shf_0$ is equivalent to the asymptotic relation \eqref{GRH equiv}, as proved in \cite{OdSm11} and \cite{OM}. Therefore, it is of interest to deduce arithmetic formulas for computation of $\lambda_F(n)$.\\

An interesting arithmetic formula for $\lambda_F(n)$ is obtained in \cite[Theorem 2.4.]{OdSm11}, where the authors prove that for $F\in\shf_0$ and positive integers $n$ the Li coefficients $\lambda_F(-n)$ can be expressed as
\begin{equation}\label{lambda of -n}
\lambda_F(-n)=m_F+ \sum_{l=1}^n {n \choose l}\gamma_F(l-1)+n\log Q_F  + \sum_{l=1}^n {n \choose l}\eta_F(l-1),
\end{equation}
where
$$
\eta_F(0)= \sum_{j=1}^r \lambda_j\psi(\lambda_j+\mu_j)
$$
and
$$
\eta_F(l-1)= \sum_{j=1}^r (-\lambda_j)^l \sum_{k=0}^{\infty}\frac{1}{(\lambda_j+\mu_j+k)^l}, \text{    for  } l\geq 2.
$$
Here, $\psi(s)= \frac{\Gamma'}{\Gamma}(s)$ denotes the digamma function and $\gamma_F(k)$ are the coefficients in the Laurent (Taylor) series expansion of $\frac{F'}{F}(s)$ around its possible pole at $s=1$. The coefficients $\gamma_F(k)$ are called the generalized Euler-Stieltjes constants of the second kind (see e.g. \cite{OdSM15} for a more detailed explanation).\\


In this paper we need a slightly different representation of the generalized Li coefficient $\lambda_F(-n)$, $n\in\N$, given in the following Lemma.

\begin{lemma} For $F\in\shf_0$ the generalized Li coefficient $\lambda_F(-n)$ for $n\in\N$ can be expressed as
\begin{equation}\label{lambda of -n true}
\lambda_F(-n)=m_F+ \sum_{l=1}^n {n \choose l}\gamma_F(l-1)+n\log Q_F  + n\sum_{j=1}^r \lambda_j\psi(\lambda_j+\mu_j) + \sum_{j=1}^r \sum_{l=2}^n {n \choose l} \frac{\lambda_j^l}{(l-1)!} \psi^{(l-1)}(\lambda_j+\mu_j).
\end{equation}
\end{lemma}
\begin{proof}
Using the relation
$$
\zeta(l,z)=\sum_{k=0}^{\infty}\frac{1}{(z+k)^l}=\frac{(-1)^l}{(l-1)!}\psi^{(l-1)}(z)
$$
between the Hurwitz zeta function $\zeta(s,z)$ and the derivatives of digamma function for $l\geq 2$ we may write the equation \eqref{lambda of -n} as \eqref{lambda of -n true}.
\end{proof}

The sum
$$
S_{\infty}(F,n):=n\log Q_F+ n\sum_{j=1}^r \lambda_j\psi(\lambda_j+\mu_j) + \sum_{j=1}^r \sum_{l=2}^n {n \choose l} \frac{\lambda_j^l}{(l-1)!} \psi^{(l-1)}(\lambda_j+\mu_j)
$$
arises from the gamma factors in the functional equation (iii'), hence it is called the archimedean part of the generalized Li coefficient $\lambda_F(-n)$. Analogously, the sum
\begin{equation} \label{S NA}
S_{ \mathrm{NA}}(F,n):=m_F+ \sum_{l=1}^n {n \choose l}\gamma_F(l-1)
\end{equation}
is called the non-archimedean part of $\lambda_F(-n)$.

In \cite{OdSm11} and \cite{OM} the authors compute the full asymptotic expansion of the archimedean part of the generalized Li coefficient $\lambda_F(-n)$ attached to $F\in\shf_0$ as $n\to\infty$, and prove that the GRH is equivalent to the asymptotic bound
\begin{equation} \label{bound S NA}
\sum_{l=1}^n {n \choose l}\re\left(\gamma_F(l-1)\right)=O_F(\sqrt{n} \log n), \text{   as   } n\to\infty.
\end{equation}

\subsection{Automorphic $L-$functions}

The (finite) automorphic $L$-function $L(s,\pi)$ attached to an irreducible
unitary cuspidal representation $\pi$ of $GL_N(\Q)$ is given for $\re(s)>1$ by the absolutely convergent product over primes $p$ of its local factors
\begin{equation} \label{autL funct}
L(s,\pi) =\prod_{p}\prod_{j=1}^{N}(1-\alpha_{p,j}(\pi)p^{-s})=\sum_{n=1}^{\infty}\frac{a_{n}(\pi)}{n^{s}}.
\end{equation}
The completed $L$-function
$$\Lambda(s,\pi)=Q(\pi)^{s/2}L_{\infty}(s,\pi)L(s,\pi),$$
where $Q(\pi)$ is the conductor of $\pi$ and
$$L_{\infty}(s,\pi)=\prod_{j=1}^{N}\Gamma_{\R}(s+k_{j}(\pi))= \prod_{j=1}^{N}\pi ^{-(s+k_j(\pi))/2}\Gamma\left(\frac{1}{2}(s+k_{j}(\pi))\right)$$
is the archimedean factor satisfies the functional equation
\begin{equation} \label{funct eq Lambda}
\Lambda(s,\pi)=\epsilon(\pi)\overline{\Lambda(1-\overline{s},\pi)}
\end{equation}
with the constant $\epsilon(\pi)$ of absolute value 1.

Using the results of \cite{Gelbart-Shahidi},
\cite{Jacque-Sh I}, \cite{Jacque-Sh II}, \cite {Rudnick-Sarnak}
and \cite{Shahidi 1}-\cite{Shahidi 4}, it is proved in \cite[Sec. 4]{OdSm11} that the $L-$function $L(s,\pi)$ belongs to the class $\shf$.

\subsection{The Chebyshev polynomials of the first and the second kind}
The Chebyshev polynomials of the first kind are defined for $x\in[-1,1]$ by the recurrence relations
$T_0(x)  = 1,\ T_1(x)  = x$ and $T_{n+1}(x)  = 2xT_n(x) - T_{n-1}(x)$, see e.g. \cite[pages 993-996]{GR07}. Putting $x=\cos(\theta)$, $\theta\in{[0,\pi]}$, we may write $T_n(x)=\cos(n\theta)$.
\vskip .06in
The generating function for the sequence $\{T_{n}(x)\}$ of the Chebyshev polynomials of the first kind is
\begin{equation} \label{gen fon for Tn}
\sum_{n=0}^{\infty}T_n(x) t^n = \frac{1-tx}{1-2tx+t^2},\quad t\in(0,1).
\end{equation}
The Chebyshev polynomials of the second kind are defined for $x\in[-1,1]$ by the recurrence relations
$U_0(x) = 1, U_1(x)  = 2x$ and $U_{n+1}(x)  = 2xU_n(x) - U_{n-1}(x)$. The  generating function for the sequence $\{U_{n}(x)\}$ of the Chebyshev polynomials of the second kind is
$$\sum_{n=0}^{\infty}U_n(x) t^n = \frac{1}{1-2 t x+t^2},\quad t\in(0,1).$$

There are many relations between the Chebyshev polynomials of the first and the second kind, see \cite{GR07}, section 8.94. In the sequel, we will use relations
\begin{equation}\label{Tn deriv}
\frac{d}{dx}T_n(x)=nU_{n-1}(x)
\end{equation}
and
\begin{equation}\label{recursion T U}
(1-x^2)U_{n-1}(x)=xT_n(x)-T_{n+1}(x), \text{   } n\geq 1.
\end{equation}
\vskip .06in

\section{Representations of $\re(\lambda_F(n))$ equivalent to the GRH}

In this section we derive representations of the real part of the $n$th generalized Li coefficient attached to $F\in\shf_0$ equivalent to the GRH for $F$.

In Theorem \ref{thm: GRH implies formulas} below we prove that the GRH is equivalent to \eqref{Re(lambda) under GRH} and \eqref{int rep re lambda}, for all positive integers $n$. Then, we prove that, under certain mild assumptions on the location of the imaginary part of the first eventual non-trivial zero of $F\in\shf$ which is off the critical line, equations \eqref{Re(lambda) under GRH} and \eqref{int rep re lambda} with $n=1$ or $n=2$ actually yield the GRH for $F$.

\begin{theorem} \label{thm: GRH implies formulas}
The GRH for $F\in\shf_0$ is equivalent to representation \eqref{Re(lambda) under GRH} of $\re(\lambda_F(n))$ for all positive integers $n$, or, equivalently, to representation \eqref{int rep re lambda} of $\re(\lambda_F(n))$ for all positive integers $n$.

\end{theorem}
Let us note here that each zero in the sum \eqref{Re(lambda) under GRH} should be taken according to its multiplicity.
\begin{proof}
First, we assume that the GRH holds true, hence each $\rho\in Z(F)$ can be represented as $\rho=1/2+i\gamma$, $\gamma \in\R$.

We start with the definition of the Li coefficient and the formula $\lambda_F(n)= \overline{\lambda_F(-n)}$, hence
\begin{eqnarray}
2\re(\lambda_F(n))&=&\lambda_F(n)+\overline{\lambda_F(n)}=\lambda_F(n)+\lambda_F(-n)\nonumber \\
&=&\lim_{T\to+\infty}\sum_{|\gamma|\leq T} \left(2-\left( 1-\frac{1}{1/2+i\gamma}\right)^n - \left( 1-\frac{1}{1/2+i\gamma}\right)^{-n} \right)\nonumber\\
&=&\sum_{\rho=1/2+i\gamma} \left(2-\left(\frac{2\gamma+i}{2\gamma-i}\right)^n - \left(\frac{2\gamma+i}{2\gamma-i}\right)^{-n} \right).\nonumber
\end{eqnarray}
It is obvious that $$\left|\left(\frac{2\gamma+i}{2\gamma-i}\right)\right|=1$$ for all $\gamma=\im(\rho)$, hence
$$\left(\frac{2\gamma+i}{2\gamma-i}\right)^n + \left(\frac{2\gamma+i}{2\gamma-i}\right)^{-n}= 2 \cos(n\theta(\gamma))=2T_n(\cos(\theta(\gamma))),$$ where $\theta(\gamma)$ is the argument of $\frac{2\gamma+i}{2\gamma-i}$, when $\gamma\geq 0$, or minus the argument of $\frac{2\gamma+i}{2\gamma-i}$, when $\gamma< 0$. (Here, we take the principal value of the argument which lies in the interval $(-\pi,\pi]$.)

In both cases, due to the parity of the cosine function, we have $$\cos(\theta(\gamma))=\re\left(\frac{2\gamma+i}{2\gamma-i}\right) = \frac{4\gamma^2 -1}{4\gamma^2 +1},$$
thus, \eqref{Re(lambda) under GRH} holds true.\\

Now, we prove the equation
\begin{equation}\label{two repr}
\sum_{\rho=\sigma + i\gamma \in Z(F)} \left( 1-T_n \left(\frac{4\gamma^2 -1}{4\gamma^2 +1} \right)\right) =16n\int\limits_{0}^{\infty} \frac{x N_F(x) }{(4x^2+1)^2} U_{n-1}\left(\frac{4x^2 -1}{4x^2 +1} \right)dx - (1-(-1)^n)N_F(0).
\end{equation}
 We start with
\begin{equation}\label{int rep lambda 1}
\sum_{\rho=\sigma + i\gamma \in Z(F)} \left( 1-T_n \left(\frac{4\gamma^2 -1}{4\gamma^2 +1} \right)\right) = \lim_{T\to\infty}\int\limits_{0}^{T} \left( 1-T_n \left(\frac{4x^2 -1}{4x^2 +1} \right)\right)dN_F(x).
\end{equation}
A simple computation shows that
\begin{equation}\label{1-Tn growth}
1-T_n \left(\frac{4x^2 -1}{4x^2 +1} \right) \sim \frac{n^2}{x^2}+O_n(x^{-4}), \text{   as   } x\to\infty,
\end{equation}
hence, integrating by parts in \eqref{int rep lambda 1}, having in mind that $N(x)=O(x\log x )$, as $x\to\infty$  and that $T_n(-1)=(-1)^n$, we deduce that
$$
\sum_{\rho=\sigma + i\gamma \in Z(F)} \left( 1-T_n \left(\frac{4\gamma^2 -1}{4\gamma^2 +1} \right)\right)= \lim_{T\to\infty}\int\limits_{0}^{T} N_F(x)d\left(T_n \left(\frac{4x^2 -1}{4x^2 +1} \right)\right) -(1-(-1)^n)N_F(0).
$$
Employing equation \eqref{Tn deriv}, we immediately conclude that \eqref{two repr} holds true.
\vskip .06in
It is left to prove the converse, i.e. that equation \eqref{Re(lambda) under GRH} for all positive $n$ implies the GRH.\\

This follows trivially from the fact that $T_n(x)= \cos(n\theta)$, for $x=\cos(\theta)$, hence the right-hand side of equation \eqref{Re(lambda) under GRH} is always non-negative. Therefore, equation \eqref{Re(lambda) under GRH} yields that $\re(\lambda_F(n))\geq 0$ for all positive $n$, a condition equivalent to the GRH. This completes the proof.
\end{proof}

In some special cases, we may deduce that even more general statement holds true. Namely, we have the following theorem treating the cases $n=1$ and $n=2$.

\begin{theorem}\label{thm. n=1, n=2}
\begin{itemize}
\item[(i)]
The GRH for $F\in\shf_{1/\sqrt{3}}$ is equivalent to the equation
$$\re(\lambda_{F}(1))=16 \int_{0}^{\infty}\frac{x}{(4x^{2}+1)^{2}} N_{F}(x)dx$$
or to the equation
\begin{equation} \label{lambda of 1}
\re(\lambda_{F}(1))=\sum_{\rho=\sigma + i\gamma \in Z(F)} \left( 1-\frac{4\gamma^2 -1}{4\gamma^2 +1}\right)= \sum_{\rho=\sigma + i\gamma \in Z(F)} \frac{2}{4\gamma^2 +1}.
\end{equation}
\item[(ii)] If, additionally, $F\in\shf_{\sqrt{6}}$, then, the GRH for $F$ is also equivalent to the equation
$$\re(\lambda_{F}(2))=64 \int_{0}^{\infty}\frac{x(4x^{2}-1)}{(4x^{2}+1)^{3}} N_{F}(x)dx$$
or to the equation
\begin{equation} \label{lambda of 2}
\re(\lambda_{F}(2))=\sum_{\rho=\sigma + i\gamma \in Z(F)} \left( 1-T_2\left(\frac{4\gamma^2 -1}{4\gamma^2 +1} \right)\right) = \sum_{\rho=\sigma + i\gamma \in Z(F)} \frac{32\gamma^2 }{(4\gamma^2 +1)^2}.
\end{equation}
\end{itemize}
\end{theorem}
\begin{proof}

From the proof of Theorem \ref{thm: GRH implies formulas} and the fact that $T_1(x)=x$ we see that, in order to prove the first statement, it is sufficient to prove that \eqref{lambda of 1} yields the GRH. Analogously, in order to prove the second statement, it is sufficient to prove that \eqref{lambda of 2} yields the GRH.\\

We will first present computations for general $n$, and then insert $n=1$ and $n=2$.

Let us denote by $\rho= \sigma + i\gamma$ the non-trivial zeros of $F$. Since the non-trivial zeros of $F$ come in pairs $\rho$ and $1-\overline{\rho}$, we have
\begin{eqnarray}
2\re(\lambda_F(n))=\lambda_F(n)+\lambda_F(-n)&=& \sum_{\rho\in Z(F)} \left( 1-\left(1-\frac{1}{\rho} \right)^n\right)+\sum_{\rho\in Z(F)} \left( 1-\left(1-\frac{1}{1-\overline{\rho}} \right)^{-n}\right)\nonumber\\
&=& \sum_{\rho\in Z(F)} \left( 2-\left(\frac{\sigma-1 +i\gamma}{\sigma+i\gamma} \right)^{n}-\left(\frac{1-\sigma +i\gamma}{-\sigma+i\gamma} \right)^{n}\right).\nonumber
\end{eqnarray}

The two terms on the right hand side of the above equation are complex conjugates, hence we get
$$
\re(\lambda_F(n))= \sum_{\rho=\sigma + i\gamma \in Z(F)} \left( 1-\left(\frac{(\sigma-1)^2 +\gamma^2}{\sigma^2 +\gamma^2} \right)^{n/2} T_n\left( \frac{\sigma(\sigma-1) +\gamma^2}{\sqrt{(\sigma^2 +\gamma^2)((\sigma-1)^2 +\gamma^2)}}\right)\right).
$$
Let us define
$$
g_{n,\gamma}(\sigma)= \left(\frac{(\sigma-1)^2 +\gamma^2}{\sigma^2 +\gamma^2} \right)^{n/2} T_n\left( \frac{\sigma(\sigma-1) +\gamma^2}{\sqrt{(\sigma^2 +\gamma^2)((\sigma-1)^2 +\gamma^2)}}\right)
$$
and let $\gamma_F>0$ be such that all eventual non-trivial zeros of $F$ that lie off the critical line have the absolute value of the imaginary part bigger than $\gamma_F$.
It is obvious that
$$
g_{n,\gamma}(1/2)= T_n\left(\frac{4\gamma^2 -1}{4\gamma^2 +1} \right),
$$
hence, it remains to prove that the equation
\begin{equation}\label{g equals Ts}
\sum_{\rho=\sigma + i\gamma \in Z(F),\, |\gamma|>\gamma_F} \left( 1-g_{n,\gamma}(\sigma)\right)= \sum_{\rho=\sigma + i\gamma \in Z(F),\, |\gamma|>\gamma_F} \left( 1- T_n\left( \frac{4\gamma^2-1}{4\gamma^2+1}\right)\right),
\end{equation}
yields that all $\sigma$ must be equal to $1/2$.

Employing the fact that $Z(F)=1-\overline{Z(F)}$ and that $g_{n,-\gamma}(\sigma)=g_{n,\gamma}(\sigma)$, we may write the sum on the left hand side of equation \eqref{g equals Ts} as
$$
\frac{1}{2}\sum_{\rho=\sigma + i\gamma\in Z(F),\, |\gamma|>\gamma_F} \left( 2- (g_{n,\gamma}(\sigma) + g_{n,\gamma}(1-\sigma))\right),$$
hence it is sufficient to prove that
$$
\frac{1}{2}\sum_{\rho=\sigma + i\gamma\in Z(F),\, |\gamma|>\gamma_F} \left( 2- (g_{n,\gamma}(\sigma) + g_{n,\gamma}(1-\sigma))\right)= \sum_{\rho=\sigma + i\gamma\in Z(F),\, |\gamma|>\gamma_F} \left( 1- g_{n,\gamma}(1/2) \right)$$

if and only if $\sigma=1/2$, for $n=1,2$.\\

Now we distinguish the two cases.
\begin{itemize}
\item[(i)] When $n=1$, from the assumption of the theorem, we have $\gamma_F=1/ \sqrt{3}$ and the second derivative of $ 1- g_{1,\gamma}(\sigma)$ is equal to a product of a positive factor and $\sigma^2-3\gamma^2$, which is negative for $\sigma\in[0,1]$ and $|\gamma|>1/\sqrt{3}$. Hence, the function $ 1- g_{1,\gamma}(\sigma)$ is strictly concave, meaning that $2- (g_{1,\gamma}(\sigma) + g_{1,\gamma}(1-\sigma)) \leq 2\left( 1- g_{1,\gamma}(1/2) \right)$ for all $\sigma\in[0,1]$, $|\gamma|>1/\sqrt{3}$, and the equality holds true if and only if $\sigma=1/2$. Therefore, equation \eqref{lambda of 1} yields the GRH.
\vskip .06in
\item[(ii)] In the case $ n=2$, the assumption of the theorem yields that $\gamma_F=\sqrt{6}$. The second derivative of $ 1- g_{2,\gamma}(\sigma)$ is equal to a product of a positive factor and the expression $-[2 (3 - 2 \sigma) \sigma^4 + 4 \gamma^2 \sigma^2 (-9 + 2 \sigma) + 6 \gamma^4 (1 + 2 \sigma)]$, which is negative for $|\gamma|>\sqrt{6}$ and $\sigma\in[0,1]$. Hence, the function $ 1- g_{2,\gamma}(\sigma)$ is strictly concave, meaning that $2- (g_{2,\gamma}(\sigma) + g_{2,\gamma}(1-\sigma)) \leq 2\left( 1- g_{2,\gamma}(1/2) \right)$ for all $\sigma\in[0,1]$, $|\gamma|>\sqrt{6}$, and the equality holds true if and only if $\sigma=1/2$. Therefore, equation \eqref{lambda of 2} yields the GRH.
\end{itemize}

The proof is complete.
\end{proof}

\begin{remark}\rm
From the recurrence relations satisfied by the Chebyshev polynomials of the first and the second kind, it is obvious that, for all $n\geq 2$ we can write $T_n(x)=2^{n-1}x + P_{n-2}(x)$ and $U_{n}(x)=2^n x + Q_{n-2}(x)$, for some polynomials $P_{n-2}(x)$ and $Q_{n-2}(x)$ of degree $n-2$. Then, a short computation using the relations \eqref{Tn deriv} and \eqref{recursion T U} yields that $(g_{n,\gamma}(\sigma))''$, the second derivative of the function $g_{n,\gamma}(\sigma)$ defined in the proof of Theorem \ref{thm. n=1, n=2}, is equal to a product of positive factors and a polynomial in $\gamma$ of degree $2n$ with positive leading coefficient. Therefore, for any positive integer $n$, there exists a real number $t(n)$ such that for all $\gamma$ satisfying the inequality $|\gamma|>t(n)$ and all $\sigma\in[0,1]$ we have $(1-g_{n,\gamma}(\sigma))''<0$.\\

This implies that for $|\gamma| >t(n)$ and $\sigma\in[0,1]$ we have $2- (g_{n,\gamma}(\sigma) + g_{n,\gamma}(1-\sigma))\leq 2\left( 1- g_{n,\gamma}(1/2) \right)$ and the equality holds true if and only if $\sigma=1/2$. Hence, proceeding analogously as in the proof of Theorem \ref{thm. n=1, n=2} we may obtain more formulas equivalent to the GRH for functions $F\in\shf_{t(n)}$. In other words, we have proved that for $F\in\shf_{t(n)}$, the equation
$$\re(\lambda_{F}(n))=16n \int_{0}^{\infty}\frac{xN_F(x)}{(4x^{2}+1)^{2}} U_{n-1}\left( \frac{(4x^{2}-1)}{(4x^{2}+1)}\right)dx$$
or, equivalently, the equation
\begin{equation*}
\re(\lambda_{F}(n))=\sum_{\rho=\sigma + i\gamma \in Z(F)} \left( 1-T_n\left(\frac{4\gamma^2 -1}{4\gamma^2 +1} \right)\right)
\end{equation*}
both yield the GRH.\\

The sequence $t(n)$ is increasing, hence the assumption that $F\in\shf_{t(n)}$ becomes more restrictive for large $n$.
\end{remark}

An immediate consequence of the first part of Theorem \ref{thm. n=1, n=2} is the following corollary.

\begin{corollary}
The GRH for $F\in\shf_{1/\sqrt{3}}$ is equivalent to the equation
$$
\re\left(\frac{\xi_F'}{\xi_F}(0) \right)= -2 \sum_{\rho=\sigma + i\gamma\in Z(F)} \frac{1}{4\gamma^2+1}
$$
or to the equation
$$
\re\left(\frac{\xi_F'}{\xi_F}(0) \right)= -16\int\limits_{0}^{\infty}\frac{x}{(4x^2+1)^2} N_F(x)dx
$$
\end{corollary}
\begin{proof}
The proof follows from equation \eqref{log der xi F} with $z=0$, which implies that
\begin{equation} \label{xiF as lambda}
\frac{\xi_F'}{\xi_F}(0)=-\left.\sum_{\rho\in Z(F)}\right.^{\ast} \frac{1}{\rho}=-\lambda_F(1)
\end{equation}
and the first part of Theorem \ref{thm. n=1, n=2}.
\end{proof}

\begin{remark}\rm
When all coefficients $a_F(n)$ in the Dirichlet series representation \eqref{ds rep} are real, the reflection principle and the functional equation imply that $Z(F)=1-Z(F)=1-\overline{Z(F)}$, hence, the generalized Li coefficients are real.

Moreover, under the GRH, it is possible to apply methods from  \cite{OOM11} (in the case of Dirichlet $L-$functions with real coefficients) and \cite{OOM14} (in the case of Hecke $L$-functions) and obtain that, for any $F\in\shf$ with real coefficients $a_F(n)$ and such that $N_F(0)=0$, we have
$$\lambda_{F}(n)=32n \int_{0}^{\infty}\frac{x}{(4x^{2}+1)^{2}} N_{F}^{+}(x)U_{n-1} \left( \frac{4x^{2}-1}{4x^{2}+1}\right)dx$$
and
$$\lambda_{F}(n)=2\sum_{\rho=\sigma + i\gamma \in Z(F),\, \gamma>0} \left( 1-T_{n}\left(\frac{4\gamma^2 -1}{4\gamma^2 +1} \right)\right),$$
where each zero in the last sum is counted according to its multiplicity.

In this  paper, we prove a more general statement, i.e., we show that the above representations of $\lambda_F(n)$ are equivalent to the GRH. Moreover, the method of proof of our Theorems \ref{thm: GRH implies formulas} and \ref{thm. n=1, n=2} is different from the methods given in \cite{OOM11} and \cite{OOM14} (under the GRH).
\end{remark}

\section{An integral formula for the generalized Li coefficients}

In this section we partially evaluate the integral in \eqref{int rep re lambda} and write equation \eqref{int rep re lambda} as a sum of integral and oscillatory part. Namely, we prove the following theorem.

\begin{theorem}\label{theorem.expression}
The GRH for $F\in\shf_0$ is equivalent to the formula
\begin{equation}\label{int rep Th 7}
\re(\lambda_F(n))= n\log Q_F +(1-(-1)^n)(2m_F-N_F(0))+ I_{2}(n)+16n \int_{0}^{+\infty}\frac{xS_F(x)}{(4x^2+1)^2} U_{n-1}\left(\frac{4x^2 -1}{4x^2 +1} \right) dx,
\end{equation}
for all positive integers $n$, where the function $S_F(x)$ is defined in Section 2.2. and
\begin{equation}\label{I2}
I_2(n)=\re\left(\sum_{j=1}^{r}\left[n\lambda_{j}\psi(\lambda_{j}+\mu_{j})\ + \sum_{k=2}^{n}{n \choose k} \frac{\lambda_{j}^k}{(k-1)!}\psi^{(k-1)}(\lambda_{j}+\mu_{j})\right]\right).
\end{equation}
In the case when $n=1$ the second sum in the above equation is equal to zero.
\end{theorem}

\begin{proof}
Let us put
$$
G_{n}(x):= 1-T_n \left(\frac{4x^2 -1}{4x^2 +1}\right).
$$
Then, using equation \eqref{N F in terms of argument} we get
$$
\re(\lambda_F(n))= -\int_{0}^{+\infty} N_F(x)G'_{n}(x)dx  -(1-(-1)^n)N_F(0)=I_{1}(n)+I_{2}(n)+I_{3}(n) -(1-(-1)^n)N_F(0),
$$
where
$$ I_{1}(n):=-\int_{0}^{+\infty} G'_{n}(x)\left(2m_{F}+\frac{2\log Q_{F}}{\pi}x\right)dx,$$
$$ I_{2}(n):=- \frac{1}{\pi} \sum_{j=1}^{r} \im\left(  \int_{0}^{+\infty} G'_{n}(x)\left(\log\Gamma\left(\lambda_{j}\left(\frac{1}{2}+ix\right)+\mu_{j}\right)+ \log\Gamma\left(\lambda_{j}\left(\frac{1}{2}+ix\right)+\overline{\mu_{j}}\right)\right)dx\right) $$
and
$$ I_{3}(n):=- \int_{0}^{+\infty} G'_{n}(x)S_F(x)dx= 16n \int_{0}^{+\infty}\frac{xS_F(x)}{(4x^2+1)^2} U_{n-1}\left(\frac{4x^2 -1}{4x^2 +1} \right) dx.$$
Therefore, in order to prove the statement, it is left to evaluate integrals $I_1(n)$ and $I_2(n)$.
\vskip .06in
We start with the generating function \eqref{gen fon for Tn} for the sequence of Chebyshev polynomials of the first kind and define
$$G(x,t)=\sum_{n=0}^{\infty}G_{n}(x)t^{n}=\frac{2t(1+t)}{(1-t)}\cdot\frac{1}{(2(1-t)x)^{2}+(1+t)^{2}}.$$
Then, $G(x,t)$ is differentiable with respect to $x\in\R$ for all $t\in(0,1)$ and $G'_x(x,t)=\sum_{n=0}^{\infty}G_{n}'(x)t^{n}$.\\

Moreover, equation \eqref{1-Tn growth} together with the recurrence relation \eqref{recursion T U} yield that
$$G_{n}'(x) = -\frac{16nx}{(4x^2+1)^2} U_{n-1}\left(\frac{4x^2 -1}{4x^2 +1} \right)=O_n\left(\frac{x}{(4x^2 +1)^2} \right),$$
as $x\to \pm\infty$, where the implied constant depends polynomially upon $n$, hence we may interchange the sum and the integral in order to deduce that
\begin{equation}\label{sum I1}
\sum_{n=0}^{\infty}I_{1}(n)t^{n}=-\int_{0}^{+\infty}G_x'(x,t)\left(2m_{F}+\frac{2\log Q_{F}}{\pi}x\right)dx
\end{equation}
for all $t\in(0,1)$.

Analogously, since $|\log \Gamma(z)| \sim |z|\log|z|,$ as $|z|\to\infty$, interchanging the sum and the integral, we get
\begin{multline}\label{sum I2}
\sum_{n=0}^{\infty}I_{2}(n)t^{n}=\\=- \frac{1}{\pi} \sum_{j=1}^{r} \im\left(  \int_{0}^{+\infty} G'_{x}(x,t)\left(\log\Gamma\left(\lambda_{j}\left(\frac{1}{2}+ix\right)+\mu_{j}\right)+ \log\Gamma\left(\lambda_{j}\left(\frac{1}{2}+ix\right)+\overline{\mu_{j}}\right)\right)dx\right),
\end{multline}
for all $t\in(0,1)$.
\vskip .06in
Employing equation \eqref{sum I1} we get
\begin{eqnarray}
\sum_{n=0}^{\infty}I_{1}(n)t^{n}&=&-2m_{F}\int_{0}^{+\infty}G_x'(x,t)dx-\frac{2\log Q_{F}}{\pi}\int_{0}^{+\infty}G_x'(x,t)xdx\nonumber\\
&=&4m_{F}\frac{t}{(1+t)(1-t)}+\log Q_{F}\frac{t}{(1-t)^{2}}\nonumber\\
&=&\sum_{n=1}^{\infty}\left(n\log Q_{F}+2m_{F}(1-(-1)^{n})\right)t^{n},\nonumber
\end{eqnarray}
hence
\begin{equation}
I_{1}(n)=n\log Q_{F}+2m_{F}(1-(-1)^{n}).
\end{equation}
Now, we derive an expression for $I_{2}(n)$. First, we note that
$$
\lim_{x\to\infty}\left(G(x,t)\log\Gamma\left(\lambda_{j}\left(\frac{1}{2}+ix\right)+\mu_{j}\right) \right)=0,
$$
for all $j=1,...,r$. Moreover, we have
\begin{equation*}
\im\left(G(0,t)\sum_{j=1}^{r}\left(\log\Gamma\left(\lambda_{j}/2+\mu_{j}\right)+
\log\Gamma\left(\lambda_{j}/2+\overline{\mu_{j}}\right)\right)\right)=0,
\end{equation*}
hence, integration by parts in \eqref{sum I2} yields to
\begin{equation*}
\sum_{n=0}^{\infty}I_{2}(n)t^{n}=\frac{1}{\pi}\sum_{j=1}^{r}\lambda_{j}\re\left(\int_{0}^{+\infty} G(x,t)\left(\psi\left(\lambda_{j}\left(\frac{1}{2}+ix\right)+\mu_{j}\right) + \psi\left(\lambda_{j}\left(\frac{1}{2}+ix\right)+\overline{\mu_{j}}\right) \right)dx\right).
\end{equation*}

Next, we use the fact that the function $G(x,t)$ is even function of $x$ to deduce that
\begin{eqnarray*}
\re\left(\int_{0}^{+\infty} G(x,t)\psi\left(\lambda_{j}\left(\frac{1}{2}+ix\right)+\overline{\mu_{j}}\right)dx\right)&=& \re\left(\int_{0}^{+\infty} G(x,t)\psi\left(\lambda_{j}\left(\frac{1}{2}-ix\right)+\mu_{j}\right)dx\right)\\
&=&\re\left(\int_{-\infty}^{0} G(x,t)\psi\left(\lambda_{j}\left(\frac{1}{2}+ix\right)+\mu_{j}\right)dx\right).
\end{eqnarray*}
Therefore,
\begin{equation}\label{sum I2 real}
\sum_{n=0}^{\infty}I_{2}(n)t^{n}=\frac{1}{\pi}\sum_{j=1}^r \lambda_j \re\left(\int_{-\infty}^{+\infty} G(x,t)\psi\left(\lambda_{j}\left(\frac{1}{2}+ix\right)+\mu_{j}\right)dx\right),
\end{equation}
for all $t\in(0,1)$.

Since the function $\psi(z)$ grows as $\log|z|$, as $|z|\to\infty$, in order to compute the integral on the right hand side of \eqref{sum I2 real}, we integrate the function $$g_j(z,t):=G(z,t)\psi\left(\lambda_{j}\left(\frac{1}{2}+iz\right)+\mu_{j}\right)$$ along the negatively oriented contour $C_R$ consisting of the segment of the real line from $-R$ to $R$ and the half circle $\{z\in\C:|z|=R \text{ and  } \arg{z} \in[-\pi,0]\}$, and then let $R\to+\infty$. Inside the contour $C_R$ (and along $C_R$), the imaginary part of the argument $z$ is non-positive, hence $$\re\left(\lambda_{j}\left(\frac{1}{2}+iz\right)+\mu_{j}\right) \geq \re(\lambda_j/2 + \mu_j)>0$$
(by the axiom (iii')). Therefore, the function $\psi\left(\lambda_{j}\left(\frac{1}{2}+iz\right)+\mu_{j}\right)$ is holomorphic inside $C_R$ and along $C_R$. The function $G(z,t)$ is holomorphic along $C_R$ and possesses a simple pole at $z_t=-\frac{1+t}{2(1-t)} i$ inside $C_R$, hence the function $g_{j}(z,t)$ is holomorphic on $C_R$ and inside $C_R$ except for a simple pole at $z_t$ with the residue
$$
\mathrm{Res}_{z=z_t}g_{j}(z,t)= \frac{it}{2(1-t)^2}\psi\left(\frac{\lambda_j}{(1-t)} + \mu_j\right).
$$
The contour $C_R$ is negatively oriented, hence, applying the calculus of residues we get
$$
\int_{C_R}G(z,t)\psi\left(\lambda_{j}\left(\frac{1}{2}+iz\right)+\mu_{j}\right)dz = -2\pi i \mathrm{Res}_{z=z_t}g_{j}(z,t) = \frac{\pi t}{(1-t)^2}\psi\left( \frac{\lambda_j}{(1-t)} + \mu_j\right).
$$

The integral over $C_R$ is equal to the sum of the integral over the segment of the real line from $-R$ to $R$ and the integral along the half-circle $|z|=R$, where $\arg(z)$ takes values from $0$ to $-\pi$. From the growth of the digamma function $\psi(z)$, as $|z|\to\infty$, it is obvious that the integral along the half-circle tends to zero as $R\to+\infty$, hence, letting $R\to+\infty$ we deduce that
$$
\int_{-\infty}^{+\infty} G(x,t)\psi\left(\lambda_{j}\left(\frac{1}{2}+ix\right)+\mu_{j}\right)dx = \frac{\pi t}{(1-t)^2} \psi\left(\frac{\lambda_j}{(1-t)}+\mu_j \right).
$$
This together with \eqref{sum I2 real} yields

\begin{equation} \label{I2 as sum of digammas}
\sum_{n=0}^{\infty}I_{2}(n)t^{n}= \frac{t}{(1-t)^2} \re\left( \sum_{j=1}^r \lambda_j  \psi\left(\frac{\lambda_j}{(1-t)}+\mu_j \right)\right),
\end{equation}
for all $t\in(0,1)$.

Now, we compute the Taylor series expansion of the right hand side of \eqref{I2 as sum of digammas} at $t=0$. We put $y_j=\lambda_j(1-t)^{-1}+\mu_j$, hence
$$
\frac{\lambda_j}{(1-t)^2}\frac{d}{dt}\log\Gamma\left(\frac{\lambda_j}{(1-t)}+\mu_j \right)= \frac{dy_j}{dt}\frac{d}{dy_j}\log\Gamma(y_j).
$$
The Taylor series expansion at $t=0$ corresponds to the Taylor series expansion at $y_j=\lambda_j+\mu_j$, hence we get
$$
\frac{d}{dy_j}\log\Gamma(y_j)= \psi(\lambda_j+\mu_j)+\sum_{k=1}^{\infty}\frac{\psi^{(k)}(\lambda_j+\mu_j)}{k!}(y_j-\lambda_j-\mu_j)= \psi(\lambda_j+\mu_j)+\sum_{k=1}^{\infty}\frac{\psi^{(k)}(\lambda_j+\mu_j)}{k!}\frac{\lambda_j^k t^k}{(1-t)^{k}}.
$$
Therefore
$$
\frac{\lambda_jt}{(1-t)^2}\frac{d}{dt}\log\Gamma\left(\frac{\lambda_j}{(1-t)}+\mu_j \right)=\frac{\lambda_jt}{(1-t)^2} \psi(\lambda_j+\mu_j) + \sum_{k=1}^{\infty}\frac{\psi^{(k)}(\lambda_j+\mu_j)}{k!}\frac{\lambda_j^{k+1} t^{k+1}}{(1-t)^{k+2}}.
$$
Using the identity
$$\frac{1}{(1-t)^{k+2}}=\sum_{i=0}^{\infty}\frac{(i+k+1)!}{i!(k+1)!}t^{i},$$
for $k\geq 0$ we obtain
\begin{multline}\label{Taylor comp}
\frac{\lambda_j t}{(1-t)^2}\frac{d}{dt}\log\Gamma\left(\frac{\lambda_j}{(1-t)}+\mu_j \right)=\\= \lambda_{j}\sum_{n=1}^{\infty}n \psi(\lambda_j+\mu_j)t^{n} + \sum_{i=0}^{\infty}\sum_{k=1}^{\infty}\frac{\psi^{(k)}(\lambda_j+\mu_j)}{k!}\lambda_j^{k+1} \frac{(i+k+1)!}{i!(k+1)!}t^{i+k+1}.
\end{multline}
Putting $k+1=l$ and taking $n=i+l$, we may write the second sum in equation \eqref{Taylor comp} as
$$
\sum_{n=2}^{\infty}\left(\sum_{l=2}^{n}\frac{\psi^{(l-1)}(\lambda_j+\mu_j)}{(l-1)!}\lambda_j^{l}{n \choose l}\right) t^{n}.
$$
Inserting the above equation in \eqref{Taylor comp}, together with \eqref{I2 as sum of digammas} yields \eqref{I2}. The proof is complete.
\end{proof}

In the case when $n=1$, from the above theorem and the first part of Theorem \ref{thm. n=1, n=2} we deduce the following generalization of the Volchkov criterion.
\begin{corollary}\label{cor.4}
The GRH for $F\in\shf_{1/\sqrt{3}}$ is equivalent to the formula
\begin{equation}\label{fla 1 corr 8}
\re(\lambda_F(1))= \log Q_F +4m_F+\re\left(\sum_{j=1}^{r}\lambda_{j}\psi(\lambda_{j}+\mu_{j}) \right)+ \int_{0}^{+\infty}\frac{xS_F(x)}{(x^2+1/4)^2} dx,
\end{equation}
or to the formula
\begin{equation}\label{fla 2 corr 8}
\int_{0}^{+\infty}\frac{xS_F(x)}{(x^2+1/4)^2} dx=-\re\left(\frac{\xi_F'}{\xi_F}(0)\right)-\log Q_F -4m_F-\re\left(\sum_{j=1}^{r}\lambda_{j}\psi(\lambda_{j}+\mu_{j}) \right).
\end{equation}
\end{corollary}
\begin{proof}
Formula \eqref{fla 1 corr 8} is an immediate consequence of \eqref{int rep Th 7} with $n=1$ and Theorem \ref{thm. n=1, n=2}. Equation \eqref{fla 2 corr 8} follows from \eqref{fla 1 corr 8} and \eqref{xiF as lambda}.
\end{proof}

\vskip .06in
\begin{example}\rm
In this example we show how to derive the  Volchkov criterion \eqref{sbm crit} for the Riemann zeta function as a special case of Corollary \ref{cor.4}.
\vskip .06in
For the Riemann zeta function, we have $r=1,\ \lambda_{1}=1/2,\ \mu_1=0,\ Q_{\zeta}=\pi^{-1/2},\ m_{\zeta}=1$, $N_{\zeta}(0)=0$ and the first non-trivial zero (which is on the critical line) has imaginary part bigger than $14$, hence assumptions of the corollary \ref{cor.4} are satisfied. Furthermore, we have $I_{2}(1)=\frac{1}{2}\psi(\frac{1}{2})=-\frac{\gamma}{2}-\log2$.
\vskip .06in
The coefficients in the Dirichlet series representation of $\zeta(s)$ are real, hence $\lambda_\zeta(1)$ is real and, according to the last paragraph in section 2.2., $S_\zeta(x)$ is equal to $\frac{2}{\pi}\arg \zeta(1/2 +ix)$. Therefore, Corollary \ref{cor.4} yields that the Riemann Hypothesis is equivalent to the statement
$$\frac{1}{\pi}\int_{0}^{+\infty}\frac{2x \arg \zeta(\frac{1}{2}+ix)}{(x^2+1/4)^2} dx=-\frac{\xi_{\zeta}'}{\xi_{\zeta}}(0)-4+\frac{\gamma}{2}+\log(2\sqrt{\pi}).$$
Now, using that $$-\frac{\xi_{\zeta}'}{\xi_{\zeta}}(0)=\frac{\gamma}{2}+1-\frac{1}{2}\log(4\pi),$$ we deduce that the Riemann Hypothesis is equivalent to \eqref{sbm crit}.
\end{example}

\vskip .10in

\section{The GRH and the generalized Euler-Stieltjes constants of the second kind}

\vskip .10in
The sum of generalized Euler--Stieltjes constants of the second kind is closely related to the Li criterion for the GRH, as it appears in the non-archimedean part \eqref{S NA} of the Li coefficient $\lambda_F(-n)$, $n\geq 1$.\\

As an immediate consequence of Theorems \ref{thm. n=1, n=2} and \ref{theorem.expression}, employing the relation $\re(\lambda_F(n))=\re(\lambda_F(-n))$, we obtain an integral formula for the sum of $\re(\gamma_F(l-1))$  appearing in \eqref{lambda of -n} and equivalent to the GHR, which we state as the following corollary.

\begin{corollary} \label{cor.5} The GRH for $F\in\shf_0$ is equivalent to the formula
\begin{eqnarray} \label{sum of gammaF}
\sum_{l=1}^n {n \choose l}\re\left(\gamma_F(l-1)\right)&=&(1-(-1)^n)(2m_F-N_F(0)) -m_F\nonumber\\
&&\ \ +\  16n \int_{0}^{+\infty}\frac{xS_F(x)}{(4x^2+1)^2} U_{n-1}\left(\frac{4x^2 -1}{4x^2 +1} \right) dx,
\end{eqnarray}
for all positive integers $n$, where $S_F(x)$ is defined in Section 2.2.

Specially, for functions $F\in\shf_{1/\sqrt{3}}$ the GRH is equivalent to the formula
$$
\re\left(\gamma_F(0)\right)= 3m_F+ \int_{0}^{+\infty}\frac{xS_F(x)}{(x^2+1/4)^2}dx.
$$
\end{corollary}
\vskip .06in

Moreover, since asymptotic relation \eqref{bound S NA} is equivalent to the GRH, from Corollary \ref{cor.5} we immediately deduce the following asymptotic integral criteria for the GRH.

\begin{corollary} \label{cor:asymptotic eq} The GRH for $F\in\shf_0$ is equivalent to the formula
\begin{equation} \label{asymptotic eq to GRH}
\int_{0}^{+\infty}\frac{xS_F(x)}{(x^2+1/4)^2} U_{n-1}\left(\frac{4x^2 -1}{4x^2 +1} \right) dx =O\left(\frac{\log n}{\sqrt{n}} \right), \text{   as   } n\to\infty.
\end{equation}
\end{corollary}

In the special case, when coefficients in the Dirichlet series expansion \eqref{ds rep} of $F\in\shf$ are real and the function $F$ possesses no non-trivial zeros on the real line, the coefficients $\gamma_F(l)$ are real and $S_F(x)=\frac{2}{\pi}\arg F(1/2+ix)$. In this case we have the following corollary.

\begin{corollary} Let $F\in\shf$ be a function with real coefficients in the Dirichlet series representation \eqref{ds rep} that has no non-trivial zeros on the real line. Then, the GRH is equivalent to the formula
\begin{equation} \label{gamma sum repr}
\sum_{l=1}^n {n \choose l}\gamma_F(l-1)=m_F(1-2(-1)^n)+\frac{32n}{\pi}\int_{0}^{+\infty}\frac{x \arg F(1/2+ix)}{(4x^2+1)^2} U_{n-1}\left(\frac{4x^2 -1}{4x^2 +1} \right) dx,
\end{equation}
for all positive integers $n$.
Under additional assumption that $F\in\shf_{1/\sqrt{3}}$, the  GRH is equivalent to the relation
\begin{equation} \label{gamma0 repr}
\gamma_F(0)=3m_F+ \frac{2}{\pi}\int_{0}^{+\infty}\frac{x \arg F(1/2+ix)}{(x^2+1/4)^2}dx.
\end{equation}
\end{corollary}
\vskip .06in
\begin{remark}\rm
Assuming that the coefficients in the Dirichlet series representation \eqref{ds rep} of the function $F$ are real, under the GRH, the above formula provides an integral representation of the coefficients $\gamma_F(l)$ appearing in the Laurent (Taylor) series expansion of the function $F'(s)/F(s)$ at $s=1$.
\vskip .06in
For example, when $n=1$, under the GRH, the constant term $\gamma_F(0)$ in the Laurent (Taylor) series expansion of the function $F'(s)/F(s)$ at $s=1$ possesses the integral representation \eqref{gamma0 repr}. Plugging in $n=2$ into the formula \eqref{gamma sum repr} we get, under GRH that
$$
\gamma_F(1)=-7m_F+\frac{64}{\pi}\int_{0}^{+\infty}\frac{x (4x^2 -3) \arg F(1/2+ix)}{(4x^2+1)^3}dx.
$$
Proceeding inductively in $n$, using formula \eqref{gamma sum repr} it is possible to obtain an integral representation of all generalized Euler-Stieltjes constants of the second kind, under the GRH.
\end{remark}
\vskip .06in
\begin{remark} \rm
Equation \eqref{gamma0 repr} is equivalent to the Volchkov criteria \eqref{sbm crit} for the Riemann Hypothesis, since $\gamma_{\zeta}(0)=\gamma$, $m_{\zeta}=1$ and $S_{\zeta}(x)=\frac{2}{\pi}\arg \zeta(1/2 +ix)$, hence Corollary \ref{cor.5} can also be viewed as a generalization of the Volchkov criteria to a larger class of functions.
\end{remark}

\section{An application to automorphic $L$-functions}
\vskip .10in
In this section, we derive relations equivalent to the GRH for
automorphic $L$-functions attached to irreducible, cuspidal unitary representations $GL_N(\Q)$.\\

In \cite{OdSm11} it was shown that the (finite) automorphic $L$-function $L(s,\pi)$, defined by the product of its local factors \eqref{autL funct} belongs to the class $\shf$. Moreover, for $F(s)=L(s,\pi)$ (which is a function in $\shf$) we have $r=N,\ Q_{F}=Q(\pi)^{1/2}\pi^{-N/2},\ \lambda_{j}=1/2,\ \mu_{j}=\frac{1}{2}k_{j}(\pi),\ j=1,...,N$ and $d_{F}=N$. Furthermore, when $N = 1$ and $\pi$ is trivial, $F(s) = L(s,\pi)$ reduces to the Riemann zeta function, hence, in this case $m_F = 1$ and when $N\neq1$ or $\pi$ is not trivial, the function $F(s)=L(s,\pi)$ is holomorphic at $s=1$, hence $m_{F} = 0$.
\vskip .06in

We put $\delta_\pi=1$ if $N = 1$ and $\pi$ is trivial and $\delta_\pi=0$, otherwise. We denote by $S_\pi(T)$ the value of the function $\frac{1}{\pi}\arg L(s,\pi)$ obtained by continuous variation along the straight lines joining points $1/2 -iT$, $2-iT$, $2+iT$ and $1/2 + iT$.

\vskip .06in
The completed $L-$function $\Lambda(s,\pi)$ is non-vanishing on the line $\re(s)=1$ (see \cite{Gelbart-Shahidi}), hence, by the functional equation \eqref{funct eq Lambda}, $0\notin Z(L(s,\pi))$, meaning that $F(s) = L(s,\pi)\in\shf_0$.
\vskip .06in
The application of Theorem \ref{theorem.expression}, Corollary \ref{cor.4}, Corollary \ref{cor.5} and Corollary \ref{cor:asymptotic eq} to $F(s)=L(s,\pi)$ yields the following corollary.

\begin{corollary}
We have\\
\begin{itemize}
\item[(i)] The GRH for the function $L(s,\pi)$ is equivalent to the formula
\begin{eqnarray}
\re(\lambda_\pi(n))&=& \frac{n}{2}\log\left(\frac{ Q(\pi)}{\pi^{N}}\right)+(1-(-1)^n)(2\delta_\pi-N_\pi(0))+ I_{2}(n)\nonumber \\
&& \ +\ 16n \int_{0}^{+\infty}\frac{xS_\pi(x)}{(4x^2+1)^2} U_{n-1}\left(\frac{4x^2 -1}{4x^2 +1} \right) dx,\nonumber
\end{eqnarray}
for all positive integers $n$, where
$$
I_2(n)=\re\left(\sum_{j=1}^{N}\left[\frac{n}{2}\psi\left(\frac{1}{2}(1+k_{j})\right)\ + \sum_{k=2}^{n}{n \choose k} \frac{1}{2^{k}(k-1)!}\psi^{(k-1)}\left(\frac{1}{2}(1+k_{j})\right)\right] \right).
$$
\item[(ii)] The GRH for $L(s,\pi)$ is equivalent to the formula
\begin{eqnarray}
\sum_{l=1}^n {n \choose l}\re\left(\gamma_\pi(l-1)\right)&=&(1-(-1)^n)(2\delta_\pi-N_\pi(0)) -\delta_\pi\nonumber\\
&&\ +\ 16n \int_{0}^{+\infty}\frac{x\ S_\pi(x)}{(4x^2+1)^2} U_{n-1}\left(\frac{4x^2 -1}{4x^2 +1} \right) dx,\nonumber
\end{eqnarray}
for all positive integers $n$, where the coefficients $\gamma_\pi(l)$ denote the Euler-Stieltjes constants of the second kind associated to $L(s,\pi)$.
\item[(iii)] The GRH for $L(s,\pi)$ is equivalent to the asymptotic formula
$$
\int_{0}^{+\infty}\frac{x\ S_\pi(x)}{(x^2+1/4)^2} U_{n-1}\left(\frac{4x^2 -1}{4x^2 +1} \right) dx = O\left( \frac{\log n}{\sqrt{n}}\right), \text{   as   } n\to\infty.
$$
\item[(iv)]
Under additional assumption that $L(s,\pi)$ possesses no zeros off the critical line with the absolute value of the imaginary part less than or equal to $1/\sqrt{3}$,  the GRH for $L(s,\pi)$ is equivalent to the equation
$$\int_{0}^{+\infty}\frac{xS_\pi(x)}{(x^2+1/4)^2} dx=\re(\lambda_\pi(1))-\frac{1}{2}\log\left(\frac{ Q(\pi)}{\pi^{N}}\right) -4\delta_\pi-\re\left(\sum_{j=1}^{N}\frac{1}{2}\psi\left(\frac{1}{2}(1+k_{j})\right)\right),$$
or, equivalently, to the equation
$$
\re\left(\gamma_{\pi}(0)\right)= 3\delta_{\pi}+16 \int_{0}^{+\infty}\frac{xS_{\pi}(x)}{(4x^2+1)^2}dx.
$$
\end{itemize}
\end{corollary}

\section{Concluding remarks}

In this section we present some problems that will be considered in a sequel to this article.

\subsection{Numerical computations}

The asymptotic relation \eqref{asymptotic eq to GRH} which is equivalent to the GRH seems to be very useful for the numerical investigations, due to oscillatory nature of the Chebyshev polynomials of the second kind. It would be interesting to see what would be results of such numerical computations in the case of Dirichlet $L-$functions associated to a certain character, or in the case of Hecke $L-$functions.

\subsection{Generalization of results to the class $\shfs$}

The class $\shf$ which is the main class of functions considered in this paper does not unconditionally contain all types of $L-$functions for which the GRH is assumed to hold true. In \cite{EOSS15}, the authors introduce the class $\shfs \supseteqq \shf$, as the set of all Dirichlet series $F(s)$ converging in some half-plane $\re (s)>\sigma_0\geq 1$, such that the meromorphic continuation of $F(s)$ to $\C$ is a meromorphic function of a finite order with at most finitely many poles, satisfying a functional equation relating values $F(s)$ with $\overline{F(\sigma_1-\overline{s})}$ up to multiplicative gamma factors and such that the logarithmic derivative $F'(s)/F(s)$ has a Dirichlet series representation converging in the half plane $\re (s)>\sigma_0\geq1$.

The assumptions posed on $F\in \s^{\sharp \flat}(\sigma_0, \sigma_1)$ imply that all its non-trivial zeros lie in the strip $\sigma_1-\sigma_0 \leq \re (s) \leq \sigma_0$. The class $\shfs$ contains products of suitable shifts of $L-$functions from $\shf$, as well as products of shifts of certain $L-$functions possessing an Euler product representation that are not in $\shf$ (such as the Rankin-Selberg $L-$functions).

We believe that results of this paper may be derived for functions in the class $\shfs$, under some additional assumptions on the gamma factors appearing in the functional equation axiom.

\subsection{The generalized $\tau-$Li coefficients}

The generalized $\tau$-Li coefficients, for $\tau\in[1,2]$ were introduced by A. Droll in \cite{Dr12} as the $\ast-$convergent series
$$\lambda_{F}(n,\tau)={\sum_{\rho\in Z(F)}}^{*}\left(1-\left(\frac{\rho}{\rho-\tau}\right)^{n}\right)$$
with the property that the inequality $\re(\lambda_F(n,\tau))\geq 0$ for all $n\in\N$ is equivalent to non-vanishing of $F\in\shf$ in the half-plane $\re(s)\geq \tau/2$.\\

We believe that it is possible to establish results similar to Theorem \ref{thm: GRH implies formulas} for the generalized $\tau-$Li coefficients in the sense that non-vanishing of $F\in\shf$ in the half-plane $\re(s)\geq \tau/2$ is equivalent to a formula similar to \eqref{Re(lambda) under GRH} or, equivalently, \eqref{int rep re lambda}. Furthermore, we expect that $\lambda_{F}(n,\tau)$ can be written as a sum of integral and oscillatory part which, in the case when $\tau=1$ coincides with the expression \eqref{int rep Th 7} obtained in Theorem \ref{theorem.expression}.

\subsection{Deriving further relations equivalent to the GRH}

Using Littlewood's theorem (see, e.g. \cite[pages. 132-133]{titchmarsh}) and starting with with relations given in Theorem \ref{thm. n=1, n=2} or with equation \eqref{fla 2 corr 8}, it is possible to obtain more relations equivalent to the GRH.
\vskip .06in
For example, integration by parts in \eqref{fla 2 corr 8} and evaluation of the function $$I_F(x)=\int_{0}^{x} \arg F(1/2+it)dt$$ using Littlewood's theorem along the rectangle with vertices $1/2,$ $1/2+T$, $1/2+T+iT$ and $1/2+iT$, after letting $T\to+\infty$ immediately yields a generalization of equation \eqref{Volch crit} to a more general class of functions.

\vskip .06in
Moreover, we think that there is a connection between equation \eqref{fla 2 corr 8} and the integral
\begin{equation}\label{BSY eq integral}
\frac{1}{2\pi}\int_{\Re(s)=1/2}\frac{\log|F(s)|}{|s|^{2}}|ds|
\end{equation}
which would yield a generalization of the integral criteria of Balazard, Saias and Yor from \cite{BSY99} to the class $\shf_0$.
\vskip .06in
Finally, we believe that for $F\in\shf$ with real Dirichlet series coefficients $a_F(n)$ and no Siegel zeros, the integral \eqref{BSY eq integral} taken along the line $\re(s)=a$, for $a\in(0,1]$,  can be represented as a mathematical expectation of a random variable $\log|F(X_a)|$, where $X_a$ is a complex-valued random variable whose imaginary part has the symmetric Cauchy distribution with scale $a$. In this way it is possible to obtain a relation equivalent to the GRH for $F$ in terms of Cesàro means of a certain ergodic transform, see \cite{EG15} for a similar result in the case when $F(s)$ is the Riemann zeta function.


\begin{thebibliography}{99}

\bibitem{AkMur2008}  A. Akbary and M. R. Murty, {\it Uniform distribution of zeros of Dirichlet series, Anatomy of integers}, CRM Proc.
Lecture Notes {\bf46}, Amer. Math. Soc., (2008), 143--158.

\bibitem{BSY99} M. Balazard, E. Saias, and M. Yor, {\it Notes sur la fonction de Riemann, 2},  Adv. Math., {\bf143} (1999), 284--287.


\bibitem{Dr12} A. D. Droll, {\it Variations of Li's criterion for an extension of the Selberg class}, PhD thesis, Queen's University Ontario, Canada. (2012).

\bibitem{EG15} L. Elaissaoui, Z. El-Abbidine Guennoun, {\it On logarithmic integrals of the Riemann zeta function and an approach to the Riemann Hypothesis by a geometric mean with respect to an ergodic transformation}, European J. Math, article published online on September 15 (2015), DOI 10.1007/s40879-015-0073-1.

\bibitem{EOSS15} A-M. Ernvall-Hyt$\ddot{o}$nen, A. Od\v{z}ak, L. Smajlovi\'c and M. Su\v{s}i\'c,  {\it On the modified Li criterion for a certain class of $L$-functions}, J. Number Theory. {\bf156} (2015), 340--367.

\bibitem{Gelbart-Shahidi}  S. S. Gelbart and F. Shahidi, {\it Boundedness of automorphic $L-$functions in vertical strips}, J. Amer. Math. Soc. {\bf14} (2001), 79--107.

\bibitem{GR07} I. S. Gradshteyn  and I. M. Ryzhik,  {\it Table of integrals}, series and products, Elsevier Academic Press, Amsterdam.  (2007).

\bibitem{HJM15} Y.-H. He, V. Jejjala and Dj. Minic, {\it From Veneziano to Riemann: A String Theory Statement of the Riemann Hypothesis}, Preprint, Arxiv: 1501.01975v2 (2015).


\bibitem{Jacque-Sh I}  H. Jacquet and J. A. Shalika, {\it On Euler products and the classification of automorphic representations I},
Amer. J. Math. {\bf103} (1981), 499--558.

\bibitem{Jacque-Sh II}  H. Jacquet and J. A. Shalika, {\it On Euler products and the classification of automorphic representations II},
Amer. J. Math. {\bf103} (1981), 777--815.

\bibitem{KaczSurvay}  J. Kaczorowski, {\it Axiomatic Theory of $L$-functions: the Selberg class}, Lecture Notes in Mathematics 1891,
Springer-Verlag, Berlin-Heidelberg. (2006), 133--209.

\bibitem{Kacz-PerelliActa} J. Kaczorowski and A. Perelli, {\it On the structure of the Selberg class, I: $0 \leq d
\leq 1$}, Acta Math. {\bf182} (1999), 207--241.

\bibitem{MAZ} K. Mazhouda, {\it The saddle-point method and the generalized Li coefficients},  Can. Math. Bull. {\bf54}, No. 2  (2011), 316--329.

\bibitem{OdSm11} A. Od\v{z}ak and L. Smajlovi\'{c}, {\it On asymptotic behavior of generalized Li coefficients in the Selberg class}, J. Number Theory. {\bf131} (2011), 519--535.

\bibitem{OdSM15} A. Od\v{z}ak and L. Smajlovi\'{c}, {\it Euler-Stieltjes constants for the Rankin-Selberg $L-$function and weighted Selberg orthogonality}, to appear in Glasnik Mat.

\bibitem{OM} S. Omar and K. Mazhouda, {\it The Li criterion and the Riemann hypothesis for the Selberg class II},  J. Number Theory. {\bf130} (4) (2010), 1109--1114.

\bibitem{OOM11} S. Omar, R. Ouni and K. Mazhouda, {\it On the zeros of Dirichlet  $L$-functions}, LMS J. Comput. Math. {\bf14} (2011), 140--154.

\bibitem{OOM14} S. Omar, R. Ouni and K. Mazhouda, {\it On the Li coefficients for the Hecke $L$-functions}, Math. Phys. Anal. Geom. {\bf17}, No. 1-2 (2014), 67--81.

\bibitem{Perelli}  A. Perelli, {\it A survey of the Selberg class of $L$-functions, part I}, Milan. J. Math. {\bf73} (2005), 19--52.


\bibitem{Rudnick-Sarnak}  Z. Rudnick and P. Sarnak, {\it Zeros of principal $L$-functions and random matrix theory}, Duke Math.
J. {\bf81} (1996), 269--322.

\bibitem{SBM12} S. K. Sekatskii, S. Beltraminelli and D. Merlini, {\it On equalities involving integrals of the logarithm of the Riemann $\zeta-$function and equivalent to the Riemann hypothesis}, Ukr. math. J. {\bf64} (2012), 247--261.

\bibitem{SBM09} S. K. Sekatskii, S. Beltraminelli and D. Merlini, {\it On equalities involving integrals of the logarithm of the Riemann $\zeta-$function and equivalent to the Riemann hypothesis II}, Preprint, Arxiv: 0904.1277 (2009).
\bibitem{Selberg}  A. Selberg, {\it Old and new conjectures and results about a class of Dirichlet series}, in Proc. Amalfi Conf.
Analytic Number Theory, eds. E. Bombieri et al., Universitia di Salerno. (1992), 367--385.

\bibitem{Shahidi 1}  F. Shahidi, {\it On certain $L-$functions}, Amer. J. Math. {\bf103} (1981), 297--255.

\bibitem{Shahidi 2}  F. Shahidi, {\it Fourier transforms of intertwinting operators and Plancherel measures for $GL\left( n\right)$},
Amer. J. Math. {\bf106} (1984), 67--111.

\bibitem{Shahidi 3}  F. Shahidi, {\it Local coefficients as Artin factors for real groups}, Duke Math. J. {\bf52} (1985), 973--1007.

\bibitem{Shahidi 4}  F. Shahidi, {\it A proof of Langlands' conjecture on Plancherel measures; complementary series for $p-$adic groups},
Ann. Math. {\bf132} (1990), 273--330.

\bibitem{Sm10} L. Smajlovi\'c, {\it On Li's criterion for the Riemann hypothesis for the Selberg class}, J. Number Theory. {\bf130} (2010), 828--851.

\bibitem{titchmarsh} E. C. Titchmarsh, {\it The theory of functions}, Oxford University press. (1952).

\bibitem{Volchkov} V. V. Volchkov, {\it On an equality equivalent to the Riemann hypothesis}, Ukr. math. J. {\bf47} (1995), 491--493.

\end{thebibliography}
\end{document}